\documentclass[smallextended,numbook,runningheads]{svjour3}   

\usepackage[whole]{bxcjkjatype}

\smartqed  
\usepackage{graphicx}
\usepackage{amsmath}
\usepackage{amsfonts}
\usepackage{amssymb}
\usepackage{mathptmx}  
\usepackage{bm}
\DeclareMathAlphabet{\mathcal}{OMS}{cmsy}{m}{n}
\usepackage{hyperref}
\usepackage{mathtools}
\mathtoolsset{showonlyrefs=true}
\DeclarePairedDelimiter\paren{\lparen}{\rparen}
\usepackage[linesnumbered,ruled]{algorithm2e}
\usepackage{color}
\usepackage{dcolumn}

\newcommand{\hess}{\mathsf{H}}

\newcommand{\bbR}{\mathbb{R}}

\newcommand{\rmd}{\mathrm{d}}

\usepackage{tikz}
\usepackage{pgfplotstable}
\usetikzlibrary{arrows,positioning,plotmarks,external,patterns,angles,
decorations.pathmorphing,backgrounds,fit,shapes,graphs,calc}

\journalname{}
\date{ \phantom{b} \vspace{45mm}\phantom{e}}

\allowdisplaybreaks
\begin{document}

\title{
	Adjoint-based exact Hessian computation
}


\author{
	Shin-ichi Ito \and  Takeru Matsuda \and Yuto Miyatake
}


\institute{
	S. Ito \at
	Earthquake Research Institute, The University of Tokyo, Tokyo, Japan \& Department of Mathematical Informatics, Graduate School of Information Science and Technology, The University of Tokyo, Tokyo, Japan\\
	\email{ito@eri.u-tokyo.ac.jp}
	\and
	T. Matsuda \at
	Department of Mathematical Informatics, Graduate School of Information Science and Technology, The University of Tokyo, Tokyo, Japan \& RIKEN Center for Brain Science, Saitama, Japan \\
	\email{matsuda@mist.i.u-tokyo.ac.jp}           
	\and
	Y. Miyatake \at
	Cybermedia Center, Osaka University, Osaka, Japan\\
	\email{miyatake@cas.cmc.osaka-u.ac.jp}
}

\date{Received: date / Accepted: date}

\maketitle

\begin{abstract}

	We consider a scalar function depending on a numerical solution of an initial value problem, and its second-derivative (Hessian) matrix for the initial value.
	The need to extract the information of the Hessian or to solve a linear system having the Hessian as a coefficient matrix arises in many research fields such as optimization, Bayesian estimation, and uncertainty quantification.
	From the perspective of memory efficiency, these tasks often employ a Krylov subspace method that does not need to hold the Hessian matrix explicitly and only requires computing the multiplication of the Hessian and a given vector.

	One of the ways to obtain an approximation of such Hessian-vector multiplication is to integrate the so-called second-order adjoint system numerically.
	However, the error in the approximation could be significant even if the numerical integration to the second-order adjoint system is sufficiently accurate.
	This paper presents a novel algorithm that computes the intended Hessian-vector multiplication exactly and efficiently.
	For this aim, we give a new concise derivation of the second-order adjoint system and show that the intended multiplication can be computed exactly by applying a particular numerical method to the second-order adjoint system. In the discussion, symplectic partitioned Runge--Kutta methods play an essential role.

\end{abstract}

\section{Introduction}
\label{sec1}

We consider an initial value problem of a $d$-dimensional time-dependent vector $x$ driven by an ordinary differential equation (ODE) of the form
\begin{align}
	\label{eq:original}
	\frac{\rmd}{\rmd t} x(t;\theta) = f(x(t;\theta)), \quad x(0;\theta) = \theta,
\end{align}
where $t$ is time, the function $f:\bbR^d\to\bbR^d$ is assumed to be sufficiently differentiable, and $\theta$ is an initial value of $x$.
Such ODE is often solved numerically.
Let $x_n (\theta)$ be the numerical solution that approximates the analytic solution at $t=t_{n}$, i.e., $x_n (\theta)\approx x(t_n;\theta)=x(nh;\theta)$, where $n$ is an integer and $h$ is a time step size.
This study is interested in numerically computing derivatives of a twice differentiable scalar function $C:\bbR^d \to \bbR$, which depends on the numerical solution at a certain time $t_{N}$, e.g., a gradient vector $\nabla_\theta C(x_N(\theta))$ and a second-derivative (Hessian) matrix $\hess_\theta C(x_N(\theta))$ of the function $C$ with respect to $\theta$.

Calculating the gradient $\nabla_\theta C(x_N(\theta))$
is often required to solve an optimization problem
\begin{align}\label{min_prob}
	\min_{\theta} \ C(x_N(\theta)).
\end{align}
One simple way of obtaining an approximation to the gradient is to integrate the system \eqref{eq:original} numerically multiple times for perturbed $\theta$.
For example
\begin{align}
	\frac{C(x_N(\theta + \Delta e_i )) - C (x_N(\theta))}{\Delta},
\end{align}
where $\Delta$ is a small scalar constant and $e_i$ is the $i$-th column of the $d$-dimensional identity matrix, can be seen as an approximation to the $i$-th component of the gradient.
However, when the dimensionality $d$ or the number of time steps $N$ is large, this approach becomes computationally expensive,
which makes it difficult to obtain a sufficiently accurate approximation.
Instead, in various fields such as optimal design in aerodynamics~\cite{gi00}, variational data assimilation in meteorology and oceanography~\cite{di86}, inversion problems in seismology~\cite{fi06}, and neural network~\cite{ch18}, the adjoint method has been used to approximate the gradient: the gradient is approximated by integrating the so-called adjoint system numerically.
This approach is more efficient than the simple approach in most cases, but the accuracy of the approximation is still limited when there are highly collected discretization errors.
More recently,
Sanz-Serna~\cite{ss16} showed that, if $x_N(\theta)$ is the solution of a Runge--Kutta method, the gradient $\nabla_\theta C(x_N(\theta))$ can be calculated exactly by solving the adjoint system with a particular choice of Runge--Kutta method:
the computed gradient coincides with the exact gradient up to round-off in floating point arithmetic.
Given a system of ODEs and Runge--Kutta method applied to the system,
the recipe presented in~\cite{ss16} finds the intended Runge--Kutta method for the adjoint system.
It is worth noting that, though the computation based on the recipe can be viewed as that using automatic differentiation with backward accumulation,
the recipe provides new insights in the context of adjoint methods and is quite useful for practitioners.

Hessian matrices also arise in several contexts.
For example, if we apply the Newton method to the problem \eqref{min_prob}, a linear system whose coefficient matrix is the Hessian  with respect to $\theta$ needs to be solved.
Further, the information of the inverse of the Hessian is used to quantify the uncertainty for the estimation in the Bayesian context~\cite{it16,th89}.
This case also requires solving a linear system whose coefficient matrix is the Hessian to calculate the inverse.

There are, however, several difficulties in solving such a linear system numerically.
As is the case with the gradient,
the simplest way of obtaining all elements of the Hessian is to integrate the system \eqref{eq:original} multiple times for perturbed initial value.
However, this approach is noticeably expensive, and further may suffer from the discretization error.
Therefore, calculating all elements of the Hessian by this simple approach is often computationally prohibitive.
If we apply a Krylov subspace method such as the conjugate gradient method  or conjugate residual method~\cite{sa03}, there is no need to have full entries of the Hessian, and instead, all we need to do is to compute a Hessian-vector multiplication, i.e., $(\hess_\theta C(x_N(\theta)) )\gamma$ for a given vector $\gamma\in\bbR^d$.
It was pointed out in~\cite{wa98,wa92} that, if $C$ is a function of the exact solution to \eqref{eq:original}, the Hessian-vector multiplication $(\hess_\theta C(x(t_N;\theta)) )\gamma$ can be obtained by solving the so-called second-order adjoint system backwardly.
This property indicates that solving the second-order adjoint system numerically gives an approximation to the intended Hessian-vector multiplication $(\hess_\theta C(x_N(\theta)) )\gamma$.
However, when the numerical solutions to the original system \eqref{eq:original} or second-order adjoint system are not sufficiently accurate,
the error between the intended Hessian-vector multiplication $(\hess_\theta C(x_N(\theta)) )\gamma$ and its approximation could be substantial.

In this paper, we extend the aforementioned technique, which was proposed by Sanz-Serna~\cite{ss16} to get the exact gradient, to calculate the exact Hessian-vector multiplication.
More precisely,
focusing on Runge--Kutta methods and their numerical solutions, we shall propose an algorithm that computes the Hessian-vector multiplication $(\hess_\theta C(x_N(\theta)) )\gamma$ exactly.
For this aim, we give a new concise derivation of the second-order adjoint system,
which makes it possible to discuss the second-order adjoint system within the framework of the conventional (first-order) adjoint system and to apply the technique~\cite{ss16}
to the second-order adjoint system.
We show that the intended Hessian-vector multiplication can be calculated by applying a particular choice of Runge--Kutta method to the  second-order adjoint system.

In Section~\ref{sec2}, we give a brief review of adjoint systems and the paper by Sanz-Serna~\cite{ss16}.
The main results are shown in Section~\ref{sec3},
where we present a new concise derivation of the second-order adjoint system and show how to compute the intended Hessian-vector multiplication exactly.
Section~\ref{sec4} is devoted to numerical experiments.
Concluding remarks are given in Section~\ref{sec5}.

\section{Preliminaries}
\label{sec2}

In this section, we give a brief review of adjoint systems and the paper by Sanz-Serna~\cite{ss16}.
In Section~\ref{subsec2-1}, we focus on the continuous case, where $C$ is a function of the exact solution to \eqref{eq:original}, and explain how the gradient $\nabla_\theta C(x(t_N; \theta))$ and the Hessian-vector multiplication $\paren*{\hess _\theta C(x(t_N;\theta))}\gamma$ are obtained based on the adjoint system and second-order adjoint system, respectively.
In Section~\ref{subsec2-2}, we explain that the gradient $\nabla_\theta C(x_N( \theta))$ can be calculated by solving the adjoint system using a particular choice of Runge--Kutta method.

\subsection{Adjoint method}
\label{subsec2-1}

Let $\overline{x}(t)$ be the solution to \eqref{eq:original} for the perturbed initial condition $\overline{x}(0) = \theta + \varepsilon$.
By linearizing the system \eqref{eq:original} at $x(t)$,
we see that as $\|\varepsilon\|\to 0$ it follows that $\overline{x}(t)	= x(t) + \delta (t) + \mathrm{o}(\|\varepsilon\|)$, where $\delta (t)$ solves the variational system
\begin{align}
	\label{vari_1}
	\frac{\rmd}{\rmd t}\delta = \nabla_x f(x) \delta .
\end{align}
Its solution satisfies $\delta (t) = (\nabla_\theta x (t;\theta)) \delta (0)$, that is, $\delta(t)$ depends linearly on $\delta (0)$.
The adjoint system of \eqref{vari_1}, which is usually introduced by using Lagrange multipliers, is given by\\
\begin{align} \label{adj_eq1}
	\frac{\rmd}{\rmd t} \lambda = - \nabla_x f(x) ^\top \lambda.
\end{align}
For the solutions to \eqref{vari_1} and \eqref{adj_eq1}, $\delta(t)^\top\lambda(t)$ is constant because
\begin{align}
	\frac{\rmd}{\rmd t}\lambda(t)^\top \delta(t)
	= \paren*{\frac{\rmd}{\rmd t}\lambda(t)}^\top \delta(t) + \lambda (t)^\top \paren*{\frac{\rmd}{\rmd t}\delta(t)} = 0.
\end{align}
Thus, we have
\begin{align} \label{ld1}
	\lambda (t_N)^\top \delta (t_N) = \lambda(0)^\top \delta (0).
\end{align}
On the other hand,
it follows that
\begin{align} \label{ld2}
	\nabla_x C(x(t_N;\theta))^\top \delta (t_N) =
	\nabla_\theta C(x(t_N;\theta))^\top \delta (0)
\end{align}
for any $\delta (0)$, because of the chain rule
$
	\nabla_\theta C(x(t_N;\theta)) = \nabla_\theta x(t_N;\theta)^\top
	\nabla_x C(x(t_N;\theta))
$
and
$\delta (t_N) = (\nabla_\theta x (t_N;\theta)) \delta (0)$.
By comparing \eqref{ld2} with \eqref{ld1}, it is concluded that solving the adjoint system \eqref{adj_eq1} backwardly with the final state $\lambda(t_N) = \nabla_x C(x(t_N;\theta))$ leads to the intended gradient at $t=0$, i.e., $\lambda(0) = \nabla_\theta C(x(t_N;\theta))$.

The second-order adjoint system reads~\cite{wa98,wa92}
\begin{align} \label{sec_adj_eq1}
	\frac{\rmd}{\rmd t}\xi=
	-\nabla_x f(x)^\top \xi - (\nabla_x (\nabla_x f(x))\delta))^\top \lambda,
\end{align}
where $\delta(t)$ is the solution to the variational system \eqref{vari_1}
and $\lambda(t)$ is the solution to the adjoint system \eqref{adj_eq1}.
In~\cite{wa92}, the second-order adjoint system is introduced as the variational system of the adjoint system \eqref{adj_eq1}.
Suppose that the initial state  for \eqref{vari_1} is $\delta(0)=\gamma$ and the final state for \eqref{adj_eq1} is $\lambda(t_N) = \nabla_x C(x(t_N;\theta))$.
Then, solving the second-order adjoint system \eqref{sec_adj_eq1} with the final state $\xi (t_N) = (\hess_x C(x(t_N;\theta)))\delta (t_N)$ gives the intended Hessian-vector multiplication at $t=0$, i.e., $\xi (0) =  (\hess_\theta C(x(t_N;\theta))) \gamma$.
We here skip the original proof of~\cite{wa98} and shall explain this property based on a new derivation of the second-order adjoint system in Section~\ref{sec3}.

\subsection{Exact gradient calculation}
\label{subsec2-2}

We consider  the discrete case, where $C$ is a function of the numerical solution to \eqref{eq:original} obtained by a Runge--Kutta method.
Sanz-Serna~\cite{ss16} showed that the gradient
$\nabla _\theta C(x_N(\theta))$ can be computed exactly by applying a particular choice of Runge--Kutta method for the adjoint system \eqref{adj_eq1}.
We briefly review the procedure.

Assume that the original system \eqref{eq:original} is discretized by an $s$-stage Runge--Kutta method
\begin{align}
	\begin{aligned}
		x_{n+1} & = x_n + h \sum_{i=1}^s b_i k_{n,i},                                  \\
		k_{n,i} & = f\paren*{X_{n,i}} \quad \left(i = 1,\dots,s\right) ,               \\
		X_{n,i} & = x_n + h\sum_{j=1}^s a_{ij} k_{n,j} \quad \left(i=1,\dots,s\right).
	\end{aligned} \label{RK_for11}
\end{align}
We discretize the adjoint system \eqref{adj_eq1} with another $s$-stage Runge--Kutta method
\begin{align}
	\begin{aligned}
		\lambda_{n+1} & = \lambda_n + h \sum_{i=1}^s B_i l_{n,i},                                   \\
		l_{n,i}       & = -\nabla_x f(X_{n,i})^\top\Lambda_{n,i} \quad \left(i=1,\dots,s\right),    \\
		\Lambda_{n,i} & = \lambda_n + h \sum_{j=1}^s A_{ij} l_{n,j} \quad \left(i=1,\dots,s\right).
	\end{aligned} \label{ad_RK}
\end{align}

In the continuous case, the adjoint system gives the gradient $\nabla_\theta C(x(t_N;\theta))$ due to the property $\lambda(t_N)^\top \delta(t_N) = \lambda(0)^\top \delta(0)$.
Therefore, in the discrete case, to obtain the exact gradient $\nabla _\theta C(x_N(\theta))$, the numerical solution to the adjoint system must satisfy $\lambda_N^\top \delta _N = \lambda_0^\top \delta_0$.
In~\cite{ss16}, it is proved that if the Runge--Kutta method for the adjoint system is chosen such that the pair of the Runge--Kutta methods for the original system and adjoint system constitute a symplectic partitioned Runge--Kutta method, the property $\lambda_N^\top \delta _N = \lambda_0^\top \delta_0$ is guaranteed and the gradient $\nabla _\theta C(x_N(\theta))$ is exactly obtained as shown in Theorem~\ref{thm:ss}.
We note that the symplecticity is a fundamental concept in numerical analysis for ODEs, and symplectic numerical  methods are well known in the context of geometric numerical integration.
For more details, we refer the reader to~\cite[Chapter VI]{hl06} (for symplectic partitioned Runge--Kutta methods, see also~\cite{as93,su93}).

\begin{theorem}[\cite{ss16}]
	\label{thm:ss}
	Let $x_1(\theta),\dots,x_N(\theta)$ be the approximate solutions (to $x(h;\theta),\dots, x(Nh;\theta)$) obtained by applying the Runge--Kutta method \eqref{RK_for11} characterized by the coefficients $a_{ij}$ and $b_i$ to \eqref{eq:original}.
	Assume that the coefficients $A_{ij}$ and $B_i$ of the Runge--Kutta method for the adjoint system \eqref{adj_eq1} satisfy the relation
	\begin{align}
		 & b_i = B_i \quad \left(i = 1,\dots,s\right),                         \\
		 & b_i A_{ij} + B_j a_{ji} = b_i B_j \quad \left(i,j=1,\dots,s\right).
	\end{align}
	Then, solving the adjoint system \eqref{adj_eq1} with $\lambda_N = \nabla_x C(x_N(\theta))$
	by using the Runge--Kutta method \eqref{ad_RK} characterized by $A_{ij}$ and $B_i$
	gives the exact gradient at $n=0$, i.e., $\lambda_0 = \nabla _\theta C(x_N(\theta))$.
\end{theorem}

The combination of RK methods for the original and
adjoint systems can be seen as a partitioned Runge--Kutta method for the coupled system.

\begin{remark}
	The conditions in Theorem~\ref{thm:ss} indicate that
	\begin{align}
		A_{ij} = b_j - \frac{b_j}{b_i}a_{ji}\quad \left(i,j=1,\dots,s\right),
	\end{align}
	which makes sense only when every weight $b_i$ is nonzero.
	However, for some Runge--Kutta methods such as the Runge--Kutta--Fehlberg method
	one or more weights $b_i$ vanish.
	For such cases, the above conditions cannot be used to find an appropriate Runge--Kutta method for the adjoint system.
	We refer the reader to Appendix in~\cite{ss16} for a workaround.
	For clarity, in this paper, we always assume that every weight $b_i$ is nonzero.
\end{remark}

\begin{remark}
	As explained in~\cite{ss16},
	the overall accuracy of the partitioned Runge--Kutta method for the coupled system may be lower than that of the Runge--Kutta method for \eqref{eq:original}.
	Such an undesirable property is called the order reduction.
	We need to take into account the reduction especially when we intend to compute $\nabla_\theta C( x(t_N;\theta))$ as accurately as possible in the context of, for example, sensitivity analysis (see, for example,~\cite{ha94,mu97} for the reduction of order conditions).
\end{remark}

We note that the Runge--Kutta method \eqref{ad_RK} for the adjoint system \eqref{adj_eq1} is equivalently rewritten as
\begin{align}
	\begin{aligned}
		\lambda_n     & = \lambda_{n+1} + h \sum_{i=1}^s B_i \bar{l}_{n,i},                                         \\
		\bar{l}_{n,i} & = \nabla_x f(X_{n,i})^\top\Lambda_{n,i} \quad \left(i=1,\dots,s\right),                     \\
		\Lambda_{n,i} & = \lambda_{n+1} + h \sum_{j=1}^s (B_j-A_{ij}) \bar{l}_{n,j} \quad \left(i=1,\dots,s\right).
	\end{aligned}
\end{align}
This expression is convenient when the adjoint system is solved backwardly.

\section{Hessian-vector multiplication}
\label{sec3}
Given an arbitrary vector $\gamma$,
we are interested in calculating the Hessian-vector multiplication $\paren*{\hess_\theta C(x_N(\theta)) }\gamma$ exactly.

In Section~\ref{subsec3-1}, we give a new concise derivation of the second-order adjoint system \eqref{sec_adj_eq1}.
The idea of the derivation plays an important role in Section~\ref{subsec3-2}, where we show how to calculate the exact Hessian-vector multiplication $\paren*{\hess_\theta C(x_N(\theta)) }\gamma$.

\subsection{Concise derivation of the second-order adjoint system}
\label{subsec3-1}

Let us couple the original system \eqref{eq:original} and the variational system \eqref{vari_1}.
This leads to the following system
\begin{align}
	\label{eq:coupled}
	\frac{\rmd}{\rmd t}
	\begin{bmatrix}
		x \\ \delta
	\end{bmatrix}
	=
	\begin{bmatrix}
		f(x) \\
		\nabla_x f(x) \delta
	\end{bmatrix},
	\quad
	\begin{bmatrix}
		x(0) \\ \delta(0)
	\end{bmatrix}
	= \begin{bmatrix}
		\theta \\ \gamma
	\end{bmatrix},
\end{align}
which can be written as
\begin{align}
	\frac{\rmd}{\rmd t} y = g(y),
	\quad y(0) =
	\begin{bmatrix}
		\theta \\ \gamma
	\end{bmatrix},
\end{align}
by introducing an augmented vector $y=[x^\top, \delta^\top]^\top$.
The adjoint system for \eqref{eq:coupled} is given by
\begin{align}
	\frac{\rmd}{\rmd t}
	\begin{bmatrix}
		\xi \\ \lambda
	\end{bmatrix}
	=
	-\begin{bmatrix}
		\nabla _x f(x) ^\top & (\nabla_x (\nabla_x f(x)\delta))^\top \\
		0                    & \nabla _x f(x) ^\top
	\end{bmatrix}
	\begin{bmatrix}
		\xi \\ \lambda
	\end{bmatrix},
\end{align}
which can also be written as
\begin{align}\label{adj_eq2}
	\frac{\rmd}{\rmd t} \phi = -\nabla_y g(y) ^\top \phi,
\end{align}
where $\phi = [\xi^\top, \lambda^\top]^\top$.
Note that the system \eqref{adj_eq2} is the combination of the second-order adjoint system \eqref{sec_adj_eq1} and the adjoint system \eqref{adj_eq1}.

As explained in Section~\ref{subsec2-1},
the Hessian-vector multiplication $(\hess_\theta C(x(t_N;\theta))) \gamma$ is obtained by solving the second-order adjoint system \eqref{sec_adj_eq1} backwardly.
This property was proved in Theorem~2 in \cite{wa98}, but we give another proof building on \eqref{adj_eq2}.

\begin{proposition}
	Let $x$, $\delta$ and $\lambda$ be the solutions to the original system  \eqref{eq:original} with the initial state $x(0) = \theta$, to the variational system \eqref{vari_1} with the initial state $\delta (0) = \gamma$ and  to the adjoint system \eqref{adj_eq1} with the final state $\lambda(t_N) = \nabla_x C(x(t_N;\theta))$, respectively.
	For the solution to the second-order adjoint system \eqref{sec_adj_eq1} with the final state $\xi (t_N) = (\hess_x C(x(t_N;\theta)))\delta (t_N)$, it follows that
	$\xi (0) =  (\hess_\theta C(x(t_N;\theta))) \gamma$.
\end{proposition}

\begin{proof}
	Let $\tilde{C}:\bbR^d\times \bbR^d \to \bbR$ be a real valued function defined by
	\begin{align} \label{C_tilde}
		\tilde{C} (x,\delta)
		=
		\nabla_x C(x) ^\top \delta.
	\end{align}
	Because
	\begin{align}
		\nabla _\theta C( x(t_N;\theta) )
		=
		\nabla_\theta  x(t_N;\theta) ^\top \nabla_x C( x(t_N;\theta) )
	\end{align}
	and
	\begin{align}
		\delta (t_N;\theta,\gamma)  = \nabla_\theta  x(t_N;\theta) \gamma,
	\end{align}
	it follows that
	\begin{align}
		\tilde{C} (x(t_N;\theta),\delta(t_N;\theta,\gamma))
		=
		\nabla_\theta C(x(t_N;\theta)) ^\top \gamma
	\end{align}
	for any vector $\gamma$.
	We note that the solution to the variational system \eqref{vari_1} depends on both $\theta$ and $\gamma$.
	Then, building on the discussion in Section~\ref{subsec2-1} we see that solving the adjoint system \eqref{adj_eq2} backwardly with the final states
	\begin{align}
		\xi (t_N) = \nabla _x \tilde{C} (x(t_N;\theta),\delta(t_N;\theta,\gamma))
		= (\hess_x C(x(t_N;\theta))) \delta (t_N;\theta,\gamma)
	\end{align}
	and
	\begin{align}
		\lambda (t_N) = \nabla_\delta \tilde{C} (x(t_N;\theta),\delta(t_N;\theta,\gamma))
		= \nabla_x  C(x(t_N;\theta))
	\end{align}
	leads to the Hessian-vector multiplication at $t=0$, i.e.,
	\begin{align}
		\xi (0) = \nabla _\theta \tilde{C} (x(t_N;\theta),\delta(t_N;\theta,\gamma))
		= (\hess_\theta C(x(t_N;\theta))) \gamma
	\end{align}
	as well as the gradient
	\begin{align}
		\lambda (0) = \nabla_\gamma \tilde{C} (x(t_N;\theta),\delta(t_N;\theta,\gamma))
		= \nabla_\theta C(x(t_N;\theta)).
	\end{align}
\end{proof}

The result of the above discussion lets the second-order adjoint system include within the framework of the first-order adjoint system.

\subsection{Exact Hessian-vector multiplication}
\label{subsec3-2}

From the discussion in Sections~\ref{subsec2-2} and~\ref{subsec3-1}, we readily see that the exact Hessian-vector multiplication $\paren*{\hess_\theta C(x_N(\theta)) }\gamma$ is obtained by solving the coupled adjoint system \eqref{adj_eq2} with a particular choice of Runge--Kutta method.

Suppose that the pair of $x$- and $\delta$-systems \eqref{eq:coupled} is discretized by a Runge--Kutta method:
\begin{align}
	\begin{aligned}
		y_{n+1} & = y_n + h \sum_{i=1}^s b_i p_{n,i},                                  \\
		p_{n,i} & = g\paren*{Y_{n,i}} \quad \left(i = 1,\dots,s\right),                \\
		Y_{n,i} & = y_n + h\sum_{j=1}^s a_{ij} p_{n,j} \quad \left(i=1,\dots,s\right),
	\end{aligned}\label{c_RK}
\end{align}
where $y_{n}=[x_{n}^\top, \delta_{n}^\top]^\top$.
We discretize the coupled adjoint system \eqref{adj_eq2}, i.e., the pair of
$\xi$- and $\lambda$-systems, by another Runge--Kutta method:
\begin{align}
	\begin{aligned}
		\phi_{n+1} & = \phi_n + h \sum_{i=1}^s B_i q_{n,i},                                   \\
		q_{n,i}    & = -\nabla_y g(Y_{n,i})^\top\Phi_{n,i} \quad \left(i=1,\dots,s\right),    \\
		\Phi_{n,i} & = \phi_n + h \sum_{j=1}^s A_{ij} q_{n,j} \quad \left(i=1,\dots,s\right),
	\end{aligned}\label{cad_RK}
\end{align}
where $\phi_{n}=[\xi_{n}^\top, \lambda_{n}^\top]^\top$.
Then,
we have the following theorem.

\begin{theorem} \label{thm:main}
	Let $y_1(\theta,\gamma),\dots,y_N(\theta,\gamma)$ be the solutions obtained by applying the Runge--Kutta method \eqref{c_RK} characterized by the coefficients $a_{ij}$ and $b_i$ to \eqref{eq:coupled}.
	Assume that the coefficients $A_{ij}$ and $B_i$ of the Runge--Kutta method for the coupled adjoint system \eqref{adj_eq2} satisfy the relation
	\begin{align}
		 & b_i = B_i \quad \left(i = 1,\dots,s\right),                         \\
		 & b_i A_{ij} + B_j a_{ji} = b_i B_j \quad \left(i,j=1,\dots,s\right).
	\end{align}
	Then, solving the coupled adjoint system \eqref{adj_eq2} with $\xi_N = (\hess_xC(x_N(\theta))) \gamma $ and  $\lambda_N = \nabla_x C(x_N(\theta))$ by using the Runge--Kutta method \eqref{cad_RK} characterized by $A_{ij}$ and $B_i$ gives the exact Hessian-vector multiplication at $n=0$, i.e., $\xi_0 =  (\hess_\theta C(x_N(\theta))) \gamma $ as well as the exact gradient
	$\lambda_0 = \nabla _\theta C(x_N(\theta))$.
\end{theorem}

We omit the proof of this theorem because it proceeds as that for Theorem~\ref{thm:ss} by considering \eqref{eq:coupled} and \eqref{adj_eq2} instead of \eqref{eq:original} and \eqref{adj_eq1}, respectively, and
considering the function \eqref{C_tilde} instead of $C(x)$.

As is the case with the gradient computation, an expression equivalent to \eqref{cad_RK}:
\begin{align}
	\begin{aligned}
		\phi_{n}      & = \phi_{n+1} + h \sum_{i=1}^s B_i \bar{q}_{n,i},                                          \\
		\bar{q}_{n,i} & = \nabla_y g(Y_{n,i})^\top\Phi_{n,i} \quad \left(i=1,\dots,s\right),                      \\
		\Phi_{n,i}    & = \phi_{n+1} + h \sum_{j=1}^s (B_{j}-A_{ij}) \bar{q}_{n,j} \quad \left(i=1,\dots,s\right)
	\end{aligned}\label{cad_RK2}
\end{align}
is convenient when the coupled adjoint system is solved backwardly.

\subsection{Actual computation procedure}

In this subsection, we summarize the actual computational procedure.

In what follows, ``RK1'' denotes a Runge--Kutta method with coefficients $a_{ij}$ and $b_i$, and ``RK2'' the Runge--Kutta method with coefficients $A_{ij}$ and $B_i$ determined by the conditions in Theorem~\ref{thm:main}.
The actual computational procedure for computing $(\hess_\theta C(x_N(\theta))) \gamma $ is as follows.

\begin{description}
	\item[Step~1] Integrate \eqref{eq:coupled} by using RK1 (see~\eqref{c_RK})
	      to obtain $x_1(\theta),\dots, x_N(\theta)$ and $\delta_1(\theta,\gamma),\dots, \delta_N (\theta,\gamma)$.
	      The computational cost for the $\delta$-equation (variational system) is usually cheaper than that for the $x$-equation.
	      For example, when $f$ is nonlinear and RK1 is implicit, we need to solve a nonlinear system for the $x$-equation in each time step, but we only have to solve a linear system for the $\delta$-equation.
	\item[Step~2] Integrate \eqref{adj_eq2} with $\xi_N = (\hess_xC(x_N(\theta))) \gamma $ and  $\lambda_N = \nabla_x C(x_N(\theta))$ backwardly by using RK2
	      (see~\eqref{cad_RK2}).
	      If RK1 is explicit (resp. implicit), the computation of this step is explicit (implicit).
	      Even in the implicit cases, this step only requires solving linear systems.
\end{description}

In the above procedure, solving the $x$-equation (original system) is usually the most computationally expensive part.

When we apply a Newton-type method to solve the optimization problem \eqref{min_prob}, we need to compute a linear system having the Hessian $\hess_xC(x_N(\theta))$ as a coefficient matrix in every Newton iteration.
A Krylov subspace method is one of the choices for solving such a linear system, and it requires computing Hessian-vector multiplications repeatedly for the same Hessian but different vectors.
We note that, when repeating the above procedure, we can skip solving the $x$-equation, which is the most computationally expensive part, and the $\lambda$-equation.

As far as the authors know, there has been no consensus for the standard choice of numerical integrators for the adjoint systems.
A simple strategy is to solve the adjoint systems as accurately as possible (by interpolating the internal stages of the $x$- and $\delta$-equations if necessary).
However, this strategy fails to obtain the exact Hessian-vector multiplication and makes the computation of the adjoint system more expensive.
We may conclude that the proposed method is the best choice among others in terms of the exactness of the Hessian-vector multiplication and the computational cost for the adjoint systems.

\section{Numerical verification}
\label{sec4}

This section validates the proposed method through three numerical experiments using (I) the simple pendulum (Section~\ref{SP}), (II) the one-dimensional Allen--Cahn equation (Section~\ref{ACE}), and (III) a one-dimensional inhomogeneous wave equation (Section~\ref{WE}).
The Allen--Cahn and wave equations are often employed as testbeds in the research fields of data assimilation and inversion problems.
The experiment (I) aims at confirming that the proposed method works well through the small-scale problem that enables us to check all of the elements in the Hessian matrix.
The performance of the proposed method in the large-scale problems are checked in the experiments (II) and (III), through an initial value problem (the experiment II) and a parameter-field inversion problem (the experiment III).

In all three experiments, we compare results of the proposed method with those of other numerical integrators.
Using the same Runge--Kutta method for  \eqref{eq:coupled}, we compare two Runge--Kutta methods for the coupled adjoint system \eqref{adj_eq2}.
One is the method determined by Theorem~\ref{thm:main}, and the other one is selected such that
it has the same number of stages and same order of accuracy as the Runge--Kutta method for \eqref{eq:coupled}.

\subsection{Simple pendulum}\label{SP}
We verify the discussion of Section~\ref{subsec3-2} by a numerical experiment for the simple pendulum problem
\begin{align} \label{pend}
	\frac{\rmd}{\rmd t}
	\begin{bmatrix}
		Q \\ P
	\end{bmatrix}
	=
	\begin{bmatrix}
		P \\ -\sin (Q)
	\end{bmatrix},
	\quad
	\theta =
	\begin{bmatrix}
		Q(0) \\ P(0)
	\end{bmatrix},
\end{align}
which is employed as a toy problem.
We discretize this system \eqref{pend} by the explicit Euler method.
The function $C$ is defined by $C(x) = C(Q,P) = Q^2 + QP + P^2 + P^4$.

The step size is set to $h=0.01$.
As a reference, we obtain an analytic Hessian $\hess_\theta C (x_5(\theta)) |_{\theta = [1,1]^{\top}}$ at $N=5$ with the help of symbolic computation\footnote{We used \textsf{SymPy} package in Julia.}.
Note that symbolic computation is possible only when $N$ is relatively small.
The result is
\begin{align} \label{exact_hess_sp5}
	\hess_\theta C (x_5(\theta)) |_{\theta = [1,1]^{\top}}=
	\begin{bmatrix}
		2.232746371638453 &  & 0.763132203549098 \\
		0.763132203549098 &  & 13.09116739376028
	\end{bmatrix}.
\end{align}
We compute $\xi_0$ using  the proposed approach:
the coupled adjoint system \eqref{adj_eq2} is solved by
\begin{align*}
	\phi_n = \phi_{n+1} + h \nabla_y g(y_{n})^\top \phi_{n+1}.
\end{align*}
We employ two vectors $[1,0]^\top$ and $[0,1]^\top$ as $\gamma$ to obtain the exact Hessian, and the result is
\begin{align}
	\begin{bmatrix}
		\underline{2.232746371638453} &  & \underline{0.763132203549098} \\
		\underline{0.76313220354909}9 &  & \underline{13.0911673937602}7
	\end{bmatrix},
\end{align}
which coincides with \eqref{exact_hess_sp5} to 14 digits (the underlines are drawn for the matched digits).
For comparison,
we solve the coupled adjoint system \eqref{adj_eq2}, i.e., the pair of the (first-order) adjoint system and second-order adjoint system, by applying the explicit Euler method backwardly (more precisely, the explicit method with the exchanges $\phi_{n+1} \leftrightarrow \phi_n$, $y_{n+1} \leftrightarrow y_n$ and $h\leftrightarrow -h$)
\begin{align} \label{adj_fm1}
	\phi_n = \phi_{n+1} + h \nabla_y g(y_{n+1})^\top \phi_{n+1}.
\end{align}
This formula is obviously explicit when the coupled adjoint system \eqref{adj_eq2} is solved backwardly in time. We then obtain the approximated Hessian
\begin{align}
	\begin{bmatrix}
		\underline{2.23}4679307434870 &  & \underline{0.7}71449812673337 \\
		\underline{0.7631}69390266670 &  & \underline{13.091}33376424467
	\end{bmatrix},
\end{align}
which differs from \eqref{exact_hess_sp5} substantially.
We also note that this matrix is no longer symmetric.

\subsection{Allen--Cahn equation}\label{ACE}

We consider an initial value problem of a time-dependent field variable $\psi(t,z)$ driven by the one-dimensional Allen--Cahn equation
\begin{align}
	\psi_{t} = \alpha\psi + \beta\psi_{zz} + \kappa \psi^3, \quad z \in (0,1)
\end{align}
under the Neumann boundary condition: $\psi_{z}(t,0) = \psi_{z}(t,1) = 0$.
In the following numerical experiments, the coefficients and initial value are set to $(\alpha,\beta,\kappa) = (10,0.001,-1)$ and $\psi(0,z) = \cos(\pi z)$.
We discretize the equation in space with $d$ grid points and the grid spacing $\Delta z$, i.e., $\Delta z = 1/(d-1)$, and apply the second-order central difference to the space second-derivative to obtain a semi-discrete scheme:
\begin{equation}
	\frac{\rmd\Psi^{(m)}}{\rmd t}  = \alpha\Psi^{(m)} + \kappa\left(\Psi^{(m)}\right)^3 + \frac{\beta}{\Delta z^2}
	\left\{\begin{array}{ll}
		2\left(\Psi^{(2)}-\Psi^{(1)}\right)   & (m=1)           \\
		\Psi^{(m+1)}-2\Psi^{(m)}+\Psi^{(m-1)} & (m=2,\dots,d-1) \\
		2\left(\Psi^{(d-1)}-\Psi^{(d)}\right) & (m=d)
	\end{array}\right.,
\end{equation}
where $\Psi\in\mathbb{R}^{d}$ is the discretized $\psi$, and we used $\bullet^{(m')}$ to describe a quantity $\bullet$ at $z=(m'-1)\Delta z$ to simplify the notation.
In the following, we solve the semi-discrete scheme and its variational system by the implicit Euler method, that is, we discretize \eqref{eq:coupled} by $y_{n+1} = y_{n} + h g(y_{n+1})$.
A cost function considered here is
\begin{align}
	C(\theta) = \| \Psi_{N}(\theta) - \Psi_{N}(\hat{\theta}) \|^{2}_{2},
\end{align}
where $\|\cdot\|_{2}$ denotes the Euclidean norm. The vector $\theta\in\mathbb{R}^{d}$ is an initial condition for the semi-discrete scheme and $\hat{\theta}$ is the discretized $\psi(0,z)$, i.e., $\hat{\theta}^{(m)}=\cos(\pi(m-1)\Delta z)$ for $m=1,\dots,d$.
The proposed method discretizes the coupled adjoint system \eqref{adj_eq2} by
\begin{align} \label{adj_4-2}
	\phi_n = \phi_{n+1} + h \nabla_y g(y_{n+1})^\top \phi_{n}.
\end{align}
We employ \eqref{adj_fm1} for comparison.
Let $\mathsf{H}$ and $\tilde{\mathsf{H}}$ be the Hessian matrices obtained by applying \eqref{adj_4-2} and \eqref{adj_fm1}, respectively.
We note that the procedure \eqref{adj_fm1} finds $\tilde{\mathsf{H}}$ uniquely since $\phi_0$ depends linearly on $\phi_N$.
The numerical experiments conducted here employ $d=150$, $h = 0.001$ and $N=20$, and show the results using $\theta^{(m)} = 1.05 \cos(\pi (m-1)\Delta z)$ for $m=1,\dots,d$.

First, we check the extent to which the symmetry is preserved or broken in the Hessian matrices by measuring ``degree of asymmetry'' defined by
\begin{align}
	\tau(\mathsf{M})=\|\mathsf{M}-\mathsf{M}^{\top}\|_{\max},
\end{align}
for a given matrix $\mathsf{M}$.
Here, $\|\cdot\|_{\max}$ denotes the maximum norm ($\|\mathsf{M}\|_{\max}=\max_{i,j} |\mathsf{M}_{ij}|$).
In this subsection, we also use the operator norm $\| \cdot \|_{\infty}$ induced by the vector maximum norm.
We observed that $\tau(\tilde{\mathsf{H}})=2.344 \times 10^{-5}$ for \eqref{adj_fm1} and $\tau(\mathsf{H})=1.518 \times 10^{-18}$ for \eqref{adj_4-2}.
This result tells us that the inappropriate discretization for the adjoint systems breaks the Hessian symmetry while the appropriate one based on the proposed method preserves the symmetry high-accurately.

Second, we check the extent to which $\mathsf{H}$ is well approximated by $\tilde{\mathsf{H}}$.
The difference is
$\| \mathsf{H} - \tilde{\mathsf{H}} \|_{\max}= 4.252\times 10^{-5}$
($\| \mathsf{H} - \tilde{\mathsf{H}} \|_{\infty}= 7.640\times 10^{-5}$,
$\| \mathsf{H} - \tilde{\mathsf{H}} \|_\infty / \| \mathsf{H} \| _\infty = 0.01660$,
and
$\| \mathsf{H} - \tilde{\mathsf{H}} \|_\infty / \| \tilde{\mathsf{H}} \| _\infty = 0.01634$),
which has the same order as $\tau(\mathsf{H})-\tau(\tilde{\mathsf{H}})$.
This implies that the difference comes from the asymmetry appeared in $\tilde{\mathsf{H}}$.

Third, we check what happens when solving a linear system $\mathsf{H} v = r$ with respect to $v$.
We employ the conjugate residual (CR) method\footnote{
	The conjugate gradient (CG) method is a method of choice when the coefficient matrix is real and symmetric.
	Note that the matrix $\mathsf{H}$ or $\tilde{\mathsf{H}}$ may have negative eigenvalues.
	The CG method still works even if the coefficient matrix has negative eigenvalues in theory, but we employ the CR method to reduce the risk of break-down.
	We show the results of the CR method only, but we note that similar results were obtained for the CG method in this problem setting.
}.
The tolerance is set to $1.0\times 10^{-8}$ for the relative residual measured by the maximum norm.
The vector $r$ is set to $r = \mathsf{H} v^\text{exact}$, where $v^\text{exact} = (1,0,\dots,0)^{\top}$.
Figure~\ref{fig:AC_exact} compares two approaches: ``proposed'' uses $\mathsf{H}$ while
``approximation'' uses $\tilde{\mathsf{H}}$.
It is observed from the top figure that
the relative residual monotonically decreases for both approaches,
but
faster convergence is observed for the proposed approach.
Note that CR method still works within double precision despite
the symmetry is broken.
However, from the bottom figure, which plots the error $\| v - v^\text{exact} \| _\infty$,
a significant error is observed for $\tilde{\mathsf{H}}$ even after the CR method itself converges,
while the proposed approach reaches the exact solution.
In the context of uncertainty quantification~\cite{it16,th89}, the information we need is not $\tilde{\mathsf{H}}^{-1}$ but $\mathsf{H}^{-1}$.
In this viewpoint, the calculation using $\tilde{\mathsf{H}}$ is problematic.
In particular, the fact that the CR method itself converges could increase the risk of overconfidence.
In contrast, the calculation based on the proposed approach seems of importance.

As a complementary study, let us investigate why the difference between the solution to $\tilde{\mathsf{H}} \tilde{v} = r$ computed by the CR method and $v^\text{exact}$ is so significant
despite of the difference between $\mathsf{H}$ and $ \tilde{\mathsf{H}}$, which is not sufficiently small in double precision but is still much smaller than $\mathrm{O}(1)$.
We compute the condition number\footnote{The condition number was calculated by using $\mathsf{cond}$ function in \textsf{LinearAlgebra} package of julia.} of $\tilde{\mathsf{H}}$, and the result is $\mathrm{cond}_\infty (\tilde{\mathsf{H}}) = 2.993\times 10^5$
($\mathrm{cond}_\infty (\mathsf{H}) = 2.696 \times 10^5$).
Because the condition number is moderate, we suspect that the CR method for $\tilde{\mathsf{H}} \tilde{v} = r$ actually finds an accurate solution to $\tilde{\mathsf{H}}\tilde{v} = r$.
In general, for the solutions to $\mathsf{H}v=r$ and $\tilde{\mathsf{H}} \tilde{v}=r$, it follows that
\begin{align} \label{pert}
	\frac{\| v  - \tilde{v} \|_{\infty}}{\| v\|_{\infty}}
	\leq \mathrm{cond}_\infty (\tilde{\mathsf{H}})
	\frac{\| \mathsf{H} - \tilde{\mathsf{H}} \|_{\infty} }{\| \tilde{\mathsf{H}}\|_{\infty}}.
\end{align}
The right hand side of \eqref{pert} is calculated to be $4.892\times 10^4$.
Since $\|v\|_{\infty} = \|v^{\text{exact}}\|_{\infty} = 1$ in the above problem setting, we see that $\| v - \tilde{v}\|_{\infty}$ could be as big as $\mathrm{O}(10^4)$,
which explains the undesirable property observed in Figure~\ref{fig:AC_exact} (bottom figure).
We note that the difference between $v$ and $\tilde{v}$ could result in the slow convergence in the Newton method (the slow convergence is discussed in Section~\ref{WE}).

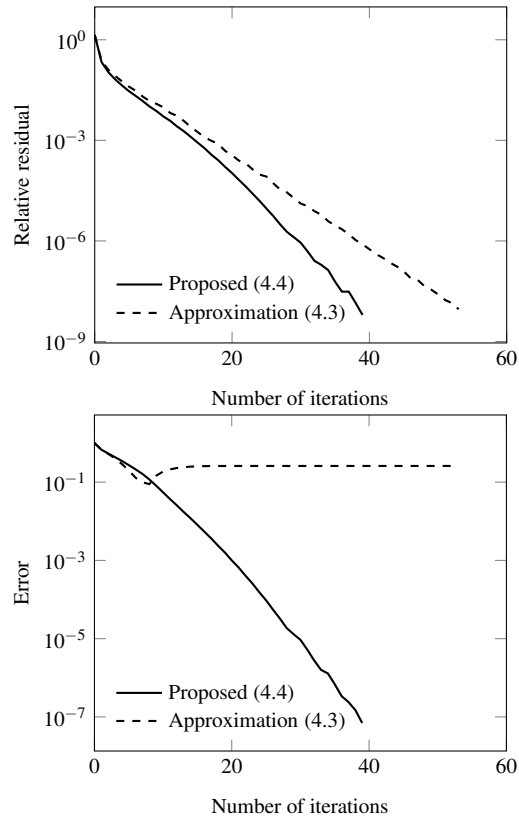
\begin{figure}[htbp]
	\centering
	\begin{tikzpicture}
		\tikzstyle{every node}=[]
		\begin{semilogyaxis}[width=7cm,
				xmax=60,xmin=0,
				xlabel={Number of iterations },ylabel={Relative residual},
				ylabel near ticks,
				legend entries={Proposed \eqref{adj_4-2}, Approximation \eqref{adj_fm1} },
				legend style={
						legend cell align=left,draw=none,fill=none,
						legend pos=south west},
			]
			\addplot[thick,
			] table {
					0 1.397e+00
					1 2.154e-01
					2 1.073e-01
					3 6.532e-02
					4 4.301e-02
					5 2.910e-02
					6 2.095e-02
					7 1.482e-02
					8 1.018e-02
					9 7.466e-03
					10 5.166e-03
					11 3.812e-03
					12 2.590e-03
					13 1.858e-03
					14 1.262e-03
					15 8.521e-04
					16 5.810e-04
					17 3.715e-04
					18 2.539e-04
					19 1.606e-04
					20 1.056e-04
					21 6.601e-05
					22 4.184e-05
					23 2.557e-05
					24 1.570e-05
					25 9.273e-06
					26 5.532e-06
					27 3.124e-06
					28 1.852e-06
					29 1.289e-06
					30 8.839e-07
					31 4.867e-07
					32 2.545e-07
					33 1.929e-07
					34 1.347e-07
					35 6.083e-08
					36 3.063e-08
					37 3.068e-08
					38 1.405e-08
					39 6.189e-09
				};
			\addplot[thick, dashed
			] table {
					0 1.397e+00
					1 2.332e-01
					2 1.262e-01
					3 8.269e-02
					4 5.684e-02
					5 3.994e-02
					6 2.981e-02
					7 2.020e-02
					8 1.488e-02
					9 1.224e-02
					10 9.658e-03
					11 6.545e-03
					12 5.376e-03
					13 3.765e-03
					14 2.500e-03
					15 1.763e-03
					16 1.225e-03
					17 9.700e-04
					18 7.783e-04
					19 4.843e-04
					20 3.616e-04
					21 2.481e-04
					22 2.059e-04
					23 1.359e-04
					24 9.569e-05
					25 8.310e-05
					26 5.248e-05
					27 3.427e-05
					28 2.789e-05
					29 1.898e-05
					30 1.301e-05
					31 1.090e-05
					32 7.722e-06
					33 6.016e-06
					34 3.863e-06
					35 2.895e-06
					36 2.228e-06
					37 1.580e-06
					38 1.057e-06
					39 7.947e-07
					40 5.669e-07
					41 4.017e-07
					42 3.097e-07
					43 2.244e-07
					44 1.729e-07
					45 1.241e-07
					46 7.788e-08
					47 6.967e-08
					48 4.522e-08
					49 3.528e-08
					50 2.546e-08
					51 1.745e-08
					52 1.395e-08
					53 9.175e-09
				};
		\end{semilogyaxis}
	\end{tikzpicture}
	\begin{tikzpicture}
		\tikzstyle{every node}=[]
		\begin{semilogyaxis}[width=7cm,
				xmax=60,xmin=0,
				xlabel={Number of iterations},ylabel={Error },
				ylabel near ticks,
				legend entries={Proposed \eqref{adj_4-2}, Approximation \eqref{adj_fm1}},
				legend style={
						legend cell align=left,draw=none,fill=none,
						legend pos=south west},
			]
			\addplot[thick,
			] table {
					0 1.000e+00
					1 6.770e-01
					2 5.376e-01
					3 4.340e-01
					4 3.475e-01
					5 2.724e-01
					6 2.118e-01
					7 1.615e-01
					8 1.170e-01
					9 8.056e-02
					10 5.472e-02
					11 3.704e-02
					12 2.548e-02
					13 1.723e-02
					14 1.194e-02
					15 8.064e-03
					16 5.439e-03
					17 3.660e-03
					18 2.381e-03
					19 1.590e-03
					20 9.994e-04
					21 6.580e-04
					22 4.053e-04
					23 2.582e-04
					24 1.545e-04
					25 9.543e-05
					26 5.535e-05
					27 3.286e-05
					28 1.868e-05
					29 1.303e-05
					30 9.353e-06
					31 5.237e-06
					32 2.776e-06
					33 1.597e-06
					34 1.291e-06
					35 6.850e-07
					36 3.401e-07
					37 2.400e-07
					38 1.494e-07
					39 6.993e-08
				};
			\addplot[thick,  dashed
			] table {
					0 1.000e+00
					1 6.795e-01
					2 5.291e-01
					3 4.052e-01
					4 2.899e-01
					5 1.992e-01
					6 1.292e-01
					7 9.902e-02
					8 8.885e-02
					9 1.459e-01
					10 1.864e-01
					11 2.115e-01
					12 2.290e-01
					13 2.408e-01
					14 2.487e-01
					15 2.533e-01
					16 2.559e-01
					17 2.574e-01
					18 2.580e-01
					19 2.584e-01
					20 2.587e-01
					21 2.588e-01
					22 2.589e-01
					23 2.589e-01
					24 2.589e-01
					25 2.589e-01
					26 2.589e-01
					27 2.589e-01
					28 2.589e-01
					29 2.589e-01
					30 2.589e-01
					31 2.589e-01
					32 2.589e-01
					33 2.589e-01
					34 2.589e-01
					35 2.589e-01
					36 2.589e-01
					37 2.589e-01
					38 2.589e-01
					39 2.589e-01
					40 2.589e-01
					41 2.589e-01
					42 2.589e-01
					43 2.589e-01
					44 2.589e-01
					45 2.589e-01
					46 2.589e-01
					47 2.589e-01
					48 2.589e-01
					49 2.589e-01
					50 2.589e-01
					51 2.589e-01
					52 2.589e-01
					53 2.589e-01
				};
		\end{semilogyaxis}
	\end{tikzpicture}
	\caption{The convergence behavior of the CR method.
		Both the relative residual and error are measured by the maximum norm.
	}
	\label{fig:AC_exact}
\end{figure}

\subsection{Wave equation}\label{WE}
We validate the proposed method through a problem to estimate an inner structure from a wave form of displacement field driven by the one-dimensional inhomogeneous wave equation
\begin{align} \label{waveeqpde}
	u_{tt} & = \left(wu_{z}\right)_{z}, \quad z\in\left[0,L\right)
\end{align}
under a periodic boundary condition, where $u\left(t,z\right)$ is a displacement field and the spatial-dependent function $w\left(z\right)$ is a ``structure field'' to be estimated.
In the following numerical experiments, the initial conditions of $u$ and its time derivative $u_{t}$ are set to $u\left(0,z\right)=16z^2\left(L-z\right)^2\slash L^4$ and $u_{t}\left(0,z\right)=0$, and we assume that $u\left(0,z\right)$ and $u_{t}\left(0,z\right)$ are known but $w$ is not, that is, the structure field $w$ is a unique control variable that determines the time evolution of the wave form of $u$.
Such a problem to estimate the unobservable inner structure field from the acoustic response of objective materials appears ubiquitously in various scientific fields such as seismology \cite{fi06} and engineering \cite{hu04,ro13}.
We discretize \eqref{waveeqpde} in space via a staggered grid with $d$ grid points and grid spacing $\Delta z$, i.e., $\Delta z = L/d$, and then we obtain
a semi-discrete scheme of \eqref{waveeqpde} given by
\begin{equation}\label{waveeq}
	\begin{aligned}
		 & \frac{\rmd U^{(m)} }{\rmd t} = V^{(m)} \quad (m=1,\dots,d)                   \\
		 & \frac{\rmd V^{(m)} }{\rmd t} = \frac{1}{\Delta z^2}
		\left\{\begin{array}{ll}
			W^{\left(\frac{3}{2}\right)} \left( U^{(2)} - U^{(1)}\right)              &                 \\
			\qquad - W^{\left(d+\frac{1}{2}\right)} \left( U^{(1)} - U^{(d)}\right)   & (m=1)           \\
			W^{\left(m+\frac{1}{2}\right)} \left( U^{(m+1)} - U^{(m)}\right)          &                 \\
			\qquad - W^{\left(m-\frac{1}{2}\right)} \left( U^{(m)} - U^{(m-1)}\right) & (m=2,\dots,d-1) \\
			W^{\left(d+\frac{1}{2}\right)} \left( U^{(1)} - U^{(d)}\right)            &                 \\
			\qquad - W^{\left(d-\frac{1}{2}\right)} \left( U^{(d)} - U^{(d-1)}\right) & (m=d)           \\
		\end{array}\right.                                        \\
		 & \frac{\rmd W^{\left(m+\frac{1}{2}\right)} }{\rmd t} = 0 \quad (m=1,\dots,d),
	\end{aligned}
\end{equation}
where $U\in\mathbb{R}^{d}$, $V\in\mathbb{R}^{d}$, and $W\in\mathbb{R}^{d}$ are the discretized $u$, $u_{t}$, and $w$, respectively, and $\bullet^{(m')}$ indicates a quantity $\bullet$ at $z=(m'-1)\Delta z$.
Throughout this section, the parameters $L$ and $\Delta z$ are fixed to $L=64$ and $\Delta z =1$ (i.e., $d=64$), and the semi-discrete scheme \eqref{waveeq} and its variational system are solved by the following 2-stage second order explicit Runge--Kutta (Heun) method:
\begin{align}
	\begin{aligned}
		y_{n+1} & = y_n + \frac{h}{2}\left(p_{n,1}+p_{n,2}\right), \\
		p_{n,1} & = g\paren*{Y_{n,1}},                             \\
		p_{n,2} & = g\paren*{Y_{n,2}},                             \\
		Y_{n,1} & = y_n,                                           \\
		Y_{n,2} & = y_n + h p_{n,1}.
	\end{aligned}
\end{align}

The cost function $C$ considered here is a sum of squared residuals between the time evolution of $U$ depending on $W$ to be estimated and that of ``observation'' of $U$:
\begin{equation}
	C\left(W\right) = \sum_{t_{n}\in T^{\text{obs}}} \| U_{n}\left(W\right)-U_{n}^{\text{obs}}\|^{2}_{2},
\end{equation}
where $T^{\text{obs}}$ is a set of the observation times, and $U^{\text{obs}}$ is the observation of $U$ obtained by solving \eqref{waveeq} assuming a ``true'' structure field $w^{\text{true}}(z)=0.5+0.25\sin\left(4\pi z\slash L\right)$.
Note that we use the Heun method with the common step size when calculating $U_{n}\left(W\right)$ and $U_{n}^{\text{obs}}$, meaning that $C$ is exactly zero when $W$ is given by the discretized $w^{\text{true}}$.
The set of observation times is fixed to $T^{\text{obs}}=\{t^{\text{obs}}\mid t^{\text{obs}}=0.2j,\;0\le j\le10,j\in\mathbb{Z}\}$ in the following experiments.

First, as well as the first experiment in Section~\ref{ACE}, we start from checking the symmetry of the Hessian.
The Hessian is evaluated at the point where $W$ is given by the discretized $w_{\text{true}}$, i.e., at the global minimum point.
Let $\mathsf{H}$ and $\tilde{\mathsf{H}}$ respectively be the Hessian evaluated by the proposed method and the Hessian obtained by applying the Heun method to the coupled adjoint system \eqref{adj_eq2} backwardly (more precisely, the Heun method with the exchanges $\phi_{n+1} \leftrightarrow \phi_n$, $y_{n+1} \leftrightarrow y_n$ and $h\leftrightarrow -h$):
\begin{align}
	\begin{aligned}
		\phi_n        & = \phi_{n+1} + \frac{h}{2}(\bar{q}_{n,1} + \bar{q}_{n,2}), \\
		\bar{q}_{n,2} & = \nabla_y g(y_{n+1})^\top \Phi_{n,2},                     \\
		\bar{q}_{n,1} & = \nabla_y g(y_n)^\top \Phi_{n,1},                         \\
		\Phi_{n,2}    & = \phi_{n+1},                                              \\
		\Phi_{n,1}    & = \phi_{n+1} + h \bar{q}_{n,2}.
	\end{aligned}\label{aHeun}
\end{align}
Although the difference between \eqref{aHeun} and the scheme of the proposed method
\begin{align}
	\begin{aligned}
		\phi_n        & = \phi_{n+1} + \frac{h}{2}(\bar{q}_{n,1} + \bar{q}_{n,2}), \\
		\bar{q}_{n,2} & = \nabla_y g(Y_{n,2})^\top \Phi_{n,2},                     \\
		\bar{q}_{n,1} & = \nabla_y g(Y_{n,1})^\top \Phi_{n,1},                     \\
		\Phi_{n,2}    & = \phi_{n+1},                                              \\
		\Phi_{n,1}    & = \phi_{n+1} + h \bar{q}_{n,2},
	\end{aligned}\label{apropose}
\end{align}
is only on the arguments of $\nabla_y g$, the small difference causes a significant difference on the symmetry of the Hessian.
Figure~\ref{fig:WavHes} shows each degree of asymmetry, $\tau(\mathsf{H})$ and $\tau(\tilde{\mathsf{H}})$, as a function of step size $h$ used for the time integrators.
Obviously the proposed method can suppress the degree of asymmetry while $\tilde{\mathsf{H}}$ cannot reproduce the symmetry when using large $h$.
The degree of asymmetry $\tau(\tilde{\mathsf{H}})$ is scaled by $h^2$, which results from the accuracy of the Heun method, meanwhile $\tau(\mathsf{H})$ does not show such convergence but a slow increase proportional to $h^{-1}$ as $h$ gets smaller mainly due to an accumulation of round-off error.
This result demonstrates that the proposed method can calculate an ``exact'' Hessian up to round-off.

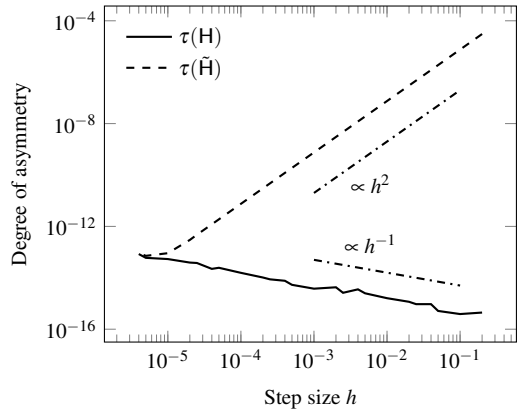
\begin{figure}[htbp]
	\centering
	\begin{tikzpicture}
		\node at(3.5,2.1) {$\propto h^2$};
		\node at(3.5,1.3) {$\propto h^{-1}$};
		\tikzstyle{every node}=[]
		\begin{loglogaxis}[width=7cm,
				xlabel={Step size $h$},ylabel={Degree of asymmetry},
				ylabel near ticks,
				legend entries={$\tau(\mathsf{H})$, $\tau(\tilde{\mathsf{H}})$ },
				legend style={
						legend cell align=left,draw=none,fill=none,
						legend pos=north west},
			]
			\addplot[thick,
			] table {
					2.000000e-01 4.440892e-16
					1.000000e-01 3.885781e-16
					5.000000e-02 5.134781e-16
					4.000000e-02 9.436896e-16
					2.500000e-02 9.436896e-16
					2.000000e-02 1.165734e-15
					1.000000e-02 1.609823e-15
					5.000000e-03 2.498002e-15
					4.000000e-03 3.552714e-15
					2.500000e-03 2.609024e-15
					2.000000e-03 4.274359e-15
					1.000000e-03 3.774758e-15
					5.000000e-04 5.329071e-15
					4.000000e-04 7.660539e-15
					2.500000e-04 8.715251e-15
					2.000000e-04 1.026956e-14
					1.000000e-04 1.559863e-14
					5.000000e-05 2.453593e-14
					4.000000e-05 2.225997e-14
					2.500000e-05 3.769207e-14
					2.000000e-05 3.913536e-14
					1.000000e-05 5.362377e-14
					5.000000e-06 5.978551e-14
					4.000000e-06 8.404388e-14
				};
			\addplot[thick, dashed
			] table {
					2.000000e-01 3.054165e-05
					1.000000e-01 7.587437e-06
					5.000000e-02 1.891973e-06
					4.000000e-02 1.210230e-06
					2.500000e-02 4.723648e-07
					2.000000e-02 3.022298e-07
					1.000000e-02 7.551447e-08
					5.000000e-03 1.887308e-08
					4.000000e-03 1.207806e-08
					2.500000e-03 4.717567e-09
					2.000000e-03 3.019152e-09
					1.000000e-03 7.547421e-10
					5.000000e-04 1.886803e-10
					4.000000e-04 1.207542e-10
					2.500000e-04 4.716852e-11
					2.000000e-04 3.018927e-11
					1.000000e-04 7.545603e-12
					5.000000e-05 1.884451e-12
					4.000000e-05 1.205813e-12
					2.500000e-05 4.744122e-13
					2.000000e-05 2.973177e-13
					1.000000e-05 9.064971e-14
					5.000000e-06 7.149836e-14
					4.000000e-06 9.309220e-14
				};
			\addplot[thick, dash dot
			] table {
					1.0e-03 2.0e-11
					1.0e-01 2.0e-7
				};
			\addplot[thick, dash dot
			] table {
					1.0e-03 5.0e-14
					1.0e-01 5.0e-15
				};
		\end{loglogaxis}
	\end{tikzpicture}
	\caption{The degree of asymmetry as a function of step size $h$.
		The solid and dashed lines indicate $\tau(\mathsf{H})$ and $\tau(\tilde{\mathsf{H}})$, respectively.
		The dash-dotted lines indicate the quadratic and inverse functions of $h$.
	}
	\label{fig:WavHes}
\end{figure}

Next, we check the speed of convergence in the optimization based on the Newton method.
When conducting a Newton-type method to optimize a nonlinear function like $C$, a linear equation $\left(\mathsf{H}+\mathsf{D}\right)\nu=-\nabla C$ needs to be solved to get a descent direction $\nu$ in each iteration step, where $\mathsf{D}$ is a method-dependent diagonal matrix \cite{ne18,li04}.
Whether a correct descent direction is obtained or not largely depends not only on the symmetry of Hessian, as seen in the third experiment of Section~\ref{ACE}, but also on the error involved in the gradient $\nabla C$, which needs to be calculated by solving an adjoint system.
This experiment sets the step size to $h=0.2$ for all of the time integrators and employs the Levenberg-Marquardt-regularized Newton (LMRN) method \cite{ne18} to optimize the cost function $C$ with an initial guess set to $W=0.5$, and then measures the performance by the number of computations of the coupled adjoint system~\eqref{adj_eq2} needed in the optimization of $C$ (We call this ``number of backward evaluations'' for simplicity).
The LMRN method employs a tolerance set to $1.0\times10^{-8}$ for the residual measured by the maximum norm and the CR method to solve the linear equations with the same tolerance as the one used in Section~\ref{ACE}.
The result of the experiment shows that, although the optimum solutions in both cases accord with the true structure field, there is a significant difference between their speeds of convergence to get the solutions.
The lines ``proposed'' and ``approximation'' in Figure~\ref{fig:WavOpt} are the convergence behaviors to attain the optimum solutions, in which \eqref{apropose} and \eqref{aHeun} are respectively used, as a function of the number of backward evaluations.
The optimization using $\mathsf{H}$ based on the proposed method is more than twice faster than that using $\tilde{\mathsf{H}}$.
We confirmed that the former is robustly faster than the latter by checking the other cases using different step sizes.
This result lets us conclude that the proposed method provides us a great benefit in the viewpoint of the speed of convergence in the nonlinear optimization.

\begin{figure}[htbp]
	\centering
	\begin{tikzpicture}
		\tikzstyle{every node}=[]
		\begin{semilogyaxis}[width=7cm,
				xmax=3000,xmin=-100,
				ymax=1e+03,
				xlabel={Number of backward evaluations},ylabel={Cost function $C$},
				ylabel near ticks,
				legend entries={Proposed \eqref{apropose}, Approximation \eqref{aHeun}},
				legend style={
						legend cell align=left,draw=none,fill=none,
						legend pos=north east},
			]
			\addplot[thick,
			] table {
					0 1.505121e+01
					75 1.050699e+01
					146 7.137063e+00
					215 4.764256e+00
					274 3.217140e+00
					329 2.277385e+00
					375 1.719373e+00
					417 1.374206e+00
					443 5.606778e-01
					516 2.397665e-01
					571 1.622927e-01
					730 3.844081e-02
					818 9.552570e-03
					903 1.909584e-03
					992 2.371911e-04
					1081 9.348913e-06
					1162 2.553896e-08
					1243 2.180437e-13
				};
			\addplot[thick, dashed
			] table {
					0 1.505121e+01
					171 1.050719e+01
					376 7.137350e+00
					566 4.764549e+00
					697 3.217384e+00
					805 2.277565e+00
					886 1.719503e+00
					952 1.374305e+00
					984 5.610419e-01
					1143 2.399016e-01
					1257 1.621748e-01
					1942 3.841895e-02
					2098 9.542329e-03
					2247 1.904435e-03
					2407 2.359495e-04
					2563 9.257386e-06
					2700 2.505573e-08
					2838 2.099089e-13
				};
		\end{semilogyaxis}
	\end{tikzpicture}
	\caption{The variation of the cost function in the optimization as a function of the number of backward evaluations.
	}
	\label{fig:WavOpt}
\end{figure}
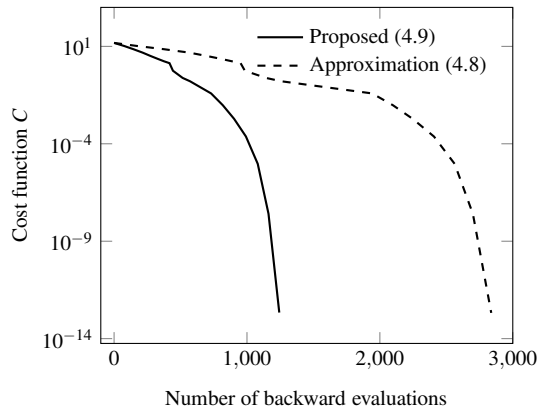

\section{Conclusion}
\label{sec5}

In this paper, we have shown a concise derivation of the second-order adjoint system, and a procedure for computing a matrix-vector multiplication exactly, where the matrix is the Hessian of a function of the numerical solution of an initial value problem with respect to the initial value.
The fact that the second-order adjoint system can be reformulated to a part of a large adjoint system is the key point to obtain the exact Hessian-vector multiplication based on the Sanz-Serna scheme.

The proposed method can be used either to obtain the exact Hessian or to solve a linear system having the Hessian as the coefficient matrix based on a Krylov subspace method.
Particularly in the latter case, the proposed method can contribute a rapid convergence since the accuracy of Hessian-vector multiplication affects the speed of convergence directly.
The importance of calculating the exact Hessian was illustrated for the Allen--Cahn and wave equations.

We plan to test the method to quantify the uncertainty for the estimation of more practical problems.
We note that in many applications an ODE system often arises from discretizing a time-dependent partial differential equation (PDE) in space.
However, discretizing PDEs in space before taking the adjoint may lead to a very strong nonphysical behavior~\cite{st97} (this issue has been partially addressed in~\cite{ti19}).
Since the proposed method does not care about the spatial discretization, we have to take the spatial discretization into account when testing the proposed method to practical problems.
Considering with such spatial discretization may provide an optimal combination of time and spatial discretizations for given problems.

\begin{acknowledgements}
	This work was triggered by discussions in the research projects of JST CREST Grant Numbers JPMJCR1761 and JPMJCR1763, JST ACT-I Grant Number JPMJPR18US, JSPS Grants-in-Aid for Scientific Research (B) Grant Number 17H01703, and JSPS Grants-in-Aid for Early-Career Scientists Grant Numbers 16K17550, 19K14671 and 19K20220.
\end{acknowledgements}

\bibliographystyle{spmpsci}      
\bibliography{references}   

\end{document}